\documentclass[12pt]{amsart}
\usepackage{amsmath,amssymb,amsfonts,amsthm,amsopn}
\usepackage{graphicx,tikz}
\usepackage[all]{xy}
\usepackage{amsmath}
\usepackage{multirow, longtable, makecell, caption, array,enumitem}
\usepackage{amssymb}
\usepackage{mathtools}
\usepackage{bookmark}
\usepackage{hyperref}

\calclayout

\newtheorem{theorem}{Theorem}[section]

\newtheorem*{conj}{Conjecture}
\newtheorem{quest}{Question}

\newtheorem{remark}[theorem]{Remark}

\theoremstyle{definition}

\usepackage{cellspace} %
\def\bR{\mathbb{R}}
\def\bC{\mathbb{C}}

\def\cF{\mathcal{F}}
\def\cS{\mathcal{S}}

\def\rd{\bR^d}

\def\rdd{\bR^{2d}}

\def\la{\langle}
\def\ra{\rangle}

\def\*b{*_{\bullet}}

\def\S0{S^0_{0,0}}

\def\Bd'{B_{\delta'}}

\def\cBd'{\bar{B}_{\delta'}}

\def\o{\omega}

\begin{document}
	
\title[The HRT conjecture and a new uncertainty principle] {A note on the HRT conjecture and a new uncertainty principle for the short-time Fourier transform}
\author{Fabio Nicola}
\address{Dipartimento di Scienze Matematiche, Politecnico di Torino, corso Duca degli Abruzzi 24, 10129 Torino, Italy}
\email{fabio.nicola@polito.it}
\author{S. Ivan Trapasso}
\address{Dipartimento di Scienze Matematiche, Politecnico di Torino, corso Duca degli Abruzzi 24, 10129 Torino, Italy}
\email{salvatore.trapasso@polito.it}
\date{\today} 

\begin{abstract}
In this note we provide a negative answer to a question raised by M. Kreisel concerning a condition on the short-time Fourier transform that would imply the HRT conjecture. In particular we provide a new type of uncertainty principle for the short-time Fourier transform which forbids the arrangement of an arbitrary ``bump with fat tail'' profile.
\end{abstract}
\maketitle
\section{Introduction}
A famous open problem in Gabor analysis is the so-called \textit{HRT conjecture}, concerning the linear independence of finitely many time-frequency shifts of a non-trivial square-integrable function \cite{hrt}. To be precise, for $x,\o \in \rd$ consider the translation and modulation operators acting on $f \in L^2(\rd)$:
\[ T_x f(t) = f(t-x), \quad M_{\o}f(t) = e^{2\pi i t\cdot \o}f(t). \] For $z=(x,\o)\in \rdd$ we say that $\pi(z)f = M_\o T_x f$ is a time-frequency shift of $f$ along $z$. The HRT conjecture can thus be stated as follows:
\begin{conj} Given $g \in L^2(\rd)\setminus\{0\}$ and a set $\Lambda$ of finitely many distinct points $z_1,\ldots,z_N \in \rdd$, the set $G(g,\Lambda)= \{\pi(z_k)g \}_{k=1}^N$ is a linearly independent set of functions in $L^2(\rd)$.
\end{conj}
As of today this somewhat basic question is still unanswered. Nevertheless, the conjecture has been proved for certain classes of functions or for special arrangements of points. We address the reader to the surveys \cite{heil speegle,heil survey}, \cite[Section 11.9]{heil book} and the paper \cite{okoudjou} for a detailed and updated state of the art on the issue. As a general remark we mention that the difficulty of the problem is witnessed by the variety of techniques involved in the known partial results, and also the surprising gap between the latter and the contexts for which nothing is known. For example, a celebrated result by Linnell \cite{linnell} states that the conjecture is true for arbitrary $g \in L^2(\rd)$ and for $\Lambda$ being a finite subset of a full-rank lattice in $\rdd$ and the proof is based on von Neumann algebras arguments. In spite of the wide range of this partial result, a solution is still lacking for smooth functions with fast decay (e.g., $g \in \cS(\bR^d)$) or for general configurations of just four points. The problem is further complicated by numerical evidence in conflict with analytic conclusions \cite{gro hrt}. 

A recent contribution by Kreisel \cite{kreisel} proves the HRT conjecture under the assumption that the distance between points in $\Lambda$ is large compared to the decay of $g$. The class of functions $g$ which are best suited for this perspective include functions with sharp descent near the origin or having a singularity away from which $g$ is bounded. 

Kreisel's paper ends with a question on the short-time Fourier transform (STFT). Recall that this is defined as \[ V_g f(x,\o) = \la f,\pi(z)g \ra = \int_{\rd} e^{-2\pi i t \cdot \o} f(t)\overline{g(t-x)}dt, \quad z=(x,\o)\in \rdd, \] for given $f,g \in L^2(\rd)$, where $\la \cdot,\cdot \ra$ denotes the inner product on $L^2(\rd)$. The STFT plays a central role in modern time-frequency analysis \cite{gro book}.
\begin{quest}\label{quest ft}
	Given $f \in L^2(\rd)$ and $R,N>0$, is there a way to design a window $g \in L^2(\rd)$ such that the ``bump with fat tail'' condition
	\begin{equation}\label{fat tail sph}
	|V_g f(z)| < \frac{|\la f,g \ra|}{N}, \quad |z|>R,
	\end{equation} holds? 
\end{quest}

From a heuristic point of view this would amount to determine a window $g$ such that $V_gf$ shows a bump near the origin and a mild decay at infinity; that is, the energy of the signal accumulates a little near the origin and then spreads on the tail (hence a \textit{fat tail}). This balance is unavoidable in view of the uncertainty principle, which forbids an arbitrary accumulation near the origin. 

A positive answer to Question \ref{quest ft} would prove the HRT conjecture by \cite[Theorem 3]{kreisel}. In fact we prove that the answer is negative as a consequence of the following result, which can be interpreted as a form of the uncertainty principle for the STFT \cite{bonami,fernandez,gro up,lieb}.

\begin{theorem}\label{maint}
	Let $g(t) = e^{-\pi t^2}$ and assume that there exist $R >0$, $N>1$ and $f \in L^2(\rd)\setminus\{0\}$ such that 
	\begin{equation}\label{fat tail cyl}
	|V_g f(x,\o)| \le \frac{|\la f,g \ra|}{N}, \quad |\o|=R.
	\end{equation}
	Then
	\begin{equation}\label{R est cyl} R > \sqrt{\frac{\log N}{\pi}}. \end{equation}
\end{theorem}

This result is indeed a negative answer to Question \ref{quest ft} since $|V_gf (x,\o)| = |V_f g(-x,-\o)|$. In fact, a stronger result can be proved in the case where the cylinder in \eqref{fat tail cyl} is replaced by a sphere.

\begin{theorem}\label{ft ball}
	Let $g(t) = e^{-\pi t^2}$ and assume that there exists $R>0$, $N>1$ and $f \in L^2(\rd)\setminus\{0\}$ such that 
	\begin{equation}\label{fat tail ball}
	|V_g f(z)| \le \frac{|\la f,g \ra|}{N}, \quad |z|=R.
	\end{equation}
	Then
	\begin{equation}\label{R est ball} R \ge \sqrt{\frac{2\log N}{\pi}}. \end{equation}
	Moreover, \eqref{fat tail ball} holds with $R=\sqrt{2\log  N / \pi}$ if and only if $f(t)= ce^{-\pi t^2}$ for some $c \in \bC\setminus\{0\}$. 
\end{theorem} 

\section{Proof of the main results and remarks}
\begin{proof}[Proof of Theorem \ref{maint}] An explicit computation shows that
	\[ |V_g f(x,-\o)| = \left| \int_{\rd} e^{2\pi i t \cdot \o}e^{-\pi(t-x)^2}f(t)dt \right| = e^{-\pi\o^2}|\Phi f(z)|, \]
	where we set 
	\[ \Phi f(z) = \int_{\rd} e^{-\pi(t-z)^2}f(t)dt, \quad z=x+i\o \in \bC^d. \]
	Notice that $\Phi f$ is an entire function on $\bC^d$, since differentiation under the integral sign is allowed. Define 
	\[ M_{a,R} = \sup_{z \in Q_{a,R}} |\Phi f(z)|, \quad Q_{a,R}=\{z=x+i\o \in \bC^d : |x|\le a, |\o| \le R \}, \] where $a>0$ will be fixed in a moment. The maximum principle \cite{narasi} implies that $|\Phi f|$ takes the value $M_{a,R}$ at some point of the boundary of $Q_{a,R}$. Since $f,g\in L^2(\mathbb{R}^d)$, $V_gf$ vanishes at infinity (e.g.\ \cite[Corollary 3.10]{ct}), so that $V_g f(x,-\o) \to 0$ for $|x|\to + \infty$, uniformly with respect to $\o\in \mathbb{R}^d$. Therefore $\Phi f(x+i\o) \to 0$ for $|x| \to +\infty$, uniformly with respect to $\o$ over compact subsets of $\rd$. This shows that for sufficiently large $a>0$ we have $|\Phi f(z_0)|=M_{a,R}$ for some point $z_0=(x_0,\o_0)$ with $|\o_0|= R$. 
	
	In view of assumption \eqref{fat tail cyl} the following estimate holds:
	\[ M_{a,R} e^{-\pi R^2} = |V_gf(z_0)| \le \frac{|\Phi f(0)|}{N}, \] where we used the identity $\la f,g\ra = V_gf(0) = \Phi f(0)$; therefore 
	\[ M_{a,R} \le \frac{e^{\pi R^2}}{N}|\Phi f(0)|. \] 
	Assume now that $R \le \sqrt{\log N / \pi}$; this would imply $M_{a,R} \le |\Phi f(0)|$ and thus $\Phi f$ would be constant on $Q_{a,R}$, hence on $\bC^d$ by analytic continuation \cite{narasi}. Since $\Phi f(x+i\o) \to 0$ for $|x| \to +\infty$ as already showed above, we could conclude that $\Phi f \equiv 0$, hence $V_gf \equiv 0$ and then $f\equiv 0$, which is a contradiction. 
\end{proof}

\begin{remark}
	Notice that Theorem \ref{maint} still holds in the case where the cylinder in \eqref{fat tail cyl} is replaced by any other cylinder obtained from the previous one by a symplectic rotation (cf.\ \cite[Sec. 2.3.2]{dg}). Indeed, if $\widehat{S}$ denotes a metaplectic operator \cite{dg} corresponding to $S \in \mathrm{Sp}(d,\bR) \cap \mathrm{O}(2d,\bR)$, condition \eqref{fat tail cyl} with $z=(x,\omega)$ replaced by $S^{-1}z$ is equivalent to
	\[ |V_g(\widehat{S}f)(x,\o)| \le \frac{|\la \widehat{S}f,g \ra|}{N}. \]
	This can be easily seen by using the covariance property \cite[Lemma 9.4.3]{gro book}
	\[
|V_g f(S^{-1}z)|=|V_{\widehat{S}g} \widehat{S}f(z)|,
	\]
the fact that $\widehat{S}$ is unitary on $L^2(\rd)$ and  that $\widehat{S}g = cg$ for some $c \in \bC$, $|c|=1$, if $g(t)=e^{-\pi t^2}$ \cite[Prop. 252]{dg}.
\end{remark}

\begin{remark} The estimate for $R$ in \eqref{R est cyl} is sharp. Consider indeed a dilated Gaussian function $f_\lambda(t) = e^{-\pi\lambda^2 t^2}$, $0 < \lambda \le 1$; a straightforward computation (see for instance \cite[Lemma 3.1]{cn}) shows that 
	\[ V_g f_\lambda (x,\o) = (1+\lambda^2)^{-d/2}e^{-2\pi i \frac{x\cdot \o}{1+\lambda^2}} e^{-\pi \frac{\lambda^2 x^2}{1+\lambda^2}} e^{-\pi \frac{\o^2}{1+\lambda^2}}. \] 
Condition \eqref{fat tail cyl} is thus satisfied if and only if 
\[ R \ge \sqrt{(1+\lambda^2) \frac{\log N}{\pi}}, \] and letting $\lambda \to 0^+$ yields the bound in \eqref{R est cyl}. 

It is worth emphasizing that there is no non-zero $f \in L^2(\rd)$ such that the optimal bound in \eqref{R est cyl} can be attained, in contrast to other uncertainty principles for the STFT. 
\end{remark} 

\begin{proof}[Proof of Theorem \ref{ft ball}]
Recall the connection between the STFT and the \textit{Bargmann transform} of a function $f \in L^2(\rd)$ \cite[Prop. 3.4.1]{gro book}:
\begin{equation}\label{barg stft} V_g f (x,-\o) = 2^{-d/4}e^{\pi i x\cdot \o} \mathcal{B}f(z) e^{-\pi |z|^2/2}, \quad z=x+i\o \in \bC^d, \end{equation} where the Bargmann transform is defined by
\[ \mathcal{B}f(z) = 2^{d/4} \int_{\rd} f(t) e^{2\pi t\cdot z - \pi t^2 - \pi z^2 /2}dt; \]
(here $g(t) = e^{-\pi t^2}$  as in the statement). This correspondence is indeed a unitary operator from $L^2(\rd)$ onto the \textit{Bargmann-Fock space} $\cF^2(\bC^d)$, i.e.\ the Hilbert space of all entire functions $F$ on $\bC^d$ such that $e^{-\pi |\cdot|^2/2}F \in L^2(\bC^d)$, cf.\ \cite[Sec. 3.4]{gro book} (see also \cite{toft1,toft2}). 

We now argue as in the proof of Theorem \ref{maint}. After setting 
\[ M_R = \sup_{z\in B_R(0)} |\mathcal{B}f(z)|, \quad B_R(0) = \{z \in \bC^d : |z|\le R \}, \] the maximum principle implies that $|\mathcal{B}f|$ takes the value $M_R$ on some point $z$ with $|z|=R$ and moreover $M_R>0$ (otherwise by analytic continuation we would have $\mathcal{B}f=0$ and therefore $f=0$).  Condition \eqref{fat tail ball} then implies \[ M_R \le \frac{e^{\pi R^2/2}}{N} |\mathcal{B}f(0)|. 
\]
If $R < \sqrt{2 \log N/\pi}$ we obtain $M_R < |\mathcal{B}f(0)|$, which is a contradiction. If $R = \sqrt{2 \log N/\pi}$ then $M_R = |\mathcal{B}f(0)|$ and therefore
 $\mathcal{B}f(z)= C$, $z \in \bC^d$, again by the maximum principle and analytic continuation, with $C\ne0$. On the other hand, a direct computation and the injectivity of the Bargmann transform show that $\mathcal{B}f(z)=1$ (hence $|V_g f (z)| = 2^{-d/4} e^{-\pi |z|^2/2}$) if and only if  $f(t)=  2^{d/4}e^{-\pi t^2}$.  This gives the last part of the claim.  
\end{proof}

\section*{Acknowledgments} The authors wish to thank Professor Elena Cordero for fruitful discussions. \\ The present research was partially supported by MIUR grant “Dipartimenti di Eccellenza” 2018–2022, CUP: E11G18000350001, DISMA, Politecnico di Torino.

\end{document}